\documentclass[12pt]{amsart}

\usepackage{amssymb}
\usepackage{amsmath}
\usepackage{amsthm,cite}

\newtheorem{theorem}{Theorem}[section]

\newenvironment{remark}[1][Remark]
           {\medbreak\noindent {\em #1. \enspace}}
           {\par \medbreak}

\makeatletter \@addtoreset{equation}{section} \makeatother

\def\ddt{\frac{d}{dt}}

\numberwithin{equation}{section}

\newcommand\lam{\lambda}
\newcommand\na{\nabla}

\newcommand\Del{\Delta}

\newcommand\Rc{\textup{Rc}}

\newcommand\SL{\text{SL}(2, \mathbb{R})}
\def\ddt{\frac{d}{dt}}

\begin{document}

\bibliographystyle{plain}
\title[]
{ Eigenvalues  under the backward Ricci flow on  locally homogeneous closed 3-manifolds}

\author{Songbo Hou}
\address{Department of Applied Mathematics, College of Science, China Agricultural
University,  Beijing, 100083, P.R. China}
\email{housb10@163.com}

\date{}
\def\baselinestretch{1}

\begin{abstract}

In this paper, we study the evolving behaviors of the first eigenvalue of the Laplace-Beltrami operator under the normalized backward Ricci flow, construct various quantities which are monotonic under the backward Ricci flow and get upper and lower bounds. We prove that in cases where the backward Ricci flow converges to a sub-Riemannian geometry after a proper rescaling, the eigenvalue evolves toward zero.
\end{abstract}
\maketitle \baselineskip 18pt
\section{Introduction}

The Ricci flow on a closed manifold is a flow of Riemannian metric $g(t)$ evolved by the equation
$$\frac{\partial g}{\partial t}=-2\Rc,\,\,g(0)=g_0,$$ where $\Rc$ is the Ricci curvature tensor of $g(t)$.
The customary normalization on 3-manifolds is setting
$\widetilde{g}(\widetilde{t})=\psi(t)g(t)$, $\widetilde{t}=\int_0^t\psi(s)ds$ with $\frac{1}{\psi}\frac{\partial\psi}{\partial t}=\frac{2r}{3}$, where $r$ is the average of the scalar curvature $R$. Then we have
$$\frac{\partial\widetilde{g}}{\partial\widetilde{t}}
=-2\widetilde{\Rc}+\frac{2\widetilde{r}}{3}\widetilde{g},\,\,\widetilde{g}(0)=g_{0}.$$
We often write it as
\begin{equation}\frac{\partial g}{\partial t}=-2\Rc+\frac{2r}{3}g,\,\,g(0)=g_0,\end{equation}
which is called the normalized Ricci flow keeping the volume constant.
In dimension 3, Hamilton \cite{Ha82} proved that
 the solution to the Ricci flow
converges to a constant curvature  metric on a 3-sphere if the Ricci curvature of initial metric is positive.

The eigenvalues of geometric operators under the Ricci flow are important to understand the geometry and the topology of  manifolds. In \cite{PG02}, Perelman proved that
the first eigenvalue of the Laplace-Beltrami operator with potential R, i.e.,$-\Del +4R$  is nondecreasing under the Ricci flow, where $R$ denotes the scalar curvature of the metric $g$. He also applied this to show that there are no nontrivial steady or expanding breathers on closed
manifolds. Later, Cao \cite{Cao07} showed that the eigenvalues of $-\Delta+\frac{R}{2}$
are nondecreasing under the Ricci flow on manifolds with nonnegative curvature
operator. Using the same techniques, Li \cite{JFL07} extended Cao's result to manifolds without nonnegative curvature operator.
Similar results hold for the first eigenvalue of $-\Delta +aR$ $ (a\geq \frac{1}{4})$ along the Ricci flow \cite{Cao08, JFL07}.
In fact, the eigenvalues $\lam$ are no longer differentiable about time $t$. If we denote by  $u$ the eigenfunction of the eigenvalue  then $\lam(u,t)=\lam(t)$. By the eigenvalue perturbation  theory,  there is a $C^1$-family of smooth eigenvalues and eigenfunctions \cite{KL06}. One can assume that the first eigenvalue $\lam(t)$ and the corresponding eigenfunction $u(x,t)$ are smooth along the Ricci flow.

Cao, the author and Ling \cite{CHL12} derived a monotonicity formula for the first eigenvalue of $-\Delta +aR$ $ (0< a \leq \frac{1}{2})$  on a closed surface under the Ricci flow and obtained various monotonicity formulae and estimates  along the normalized Ricci flow. Our results indicate that although it is difficult to get better estimates for the eigenvalue under the Ricci flow, one can get interesting results if the problem can be dealt by ODE techniques.

For the Laplace-Beltrami operator, Ma \cite{Ma06} proved that  the first eigenvalue on domains with Dirichlet boundary condition is nondecreasing along the Ricci flow.
Ling \cite{Ling07} got a sharp bound of the first eigenvalue under the normalized Ricci flow, and proved that an appropriate multiple is monotonic.  The author \cite{Hou} considered  the eigenvalue of the Laplace-Beltrami operator under the normalized Ricci flow on locally homogeneous closed 3-manifolds, constructed various monotonic quantities and got estimates for upper and lower bounds.

For $p$-Laplace operator, Wu, Wang and Zheng \cite{WWZ10} proved that the first $p$-eigenvalue is strictly increasing along  the Ricci flow under some curvature assumption, and constructed various monotonic quantities  on closed Riemannian surface. Wu  \cite{Wu11} proved that the first eigenvalue monotoncity  of the $p$-Laplace operator along the Ricci flow on closed Riemannian manifolds under some different curvature assumptions.

 A Riemannian manifold $(M,g)$ is called to be locally homogeneous if for every two points $x,y\in M$, there are neighborhoods $U$ of $x$ and $V$ of $y$, and an isometry $\phi$
from $(U,g|_U)$ to  $(V, g|_V)$ with $\phi (x)=y$.  Furthermore, $(M,g)$ is called to be homogeneous if the isometry group is transitive i.e., $U=V=M$ for all $x$ and $y$. A result of Singer \cite{Sin} told us  that the universal cover of a locally homogeneous manifold is homogeneous. We may study the Ricci flow of homogeneous models instead of locally homogeneous manifolds since Ricci flow commutes with the cover map.

The locally homogeneous 3-manifolds contain nine classes  which can be divided into two sets. The first set consists of classes  $H(3)$, $H(2)\times \mathbb{R}$ and $\text{SO(3)}\times  \mathbb{R}$, where $H(n)$ denotes the group of isometries of hyperbolic n-space. The second set includes $\mathbb{R}^3$, $\text{SU(2)}$, $\SL$, $\text{Heisenberg}$, $E(1,1)$ (the group of isometries of the plane with flat Lorentz metric) and $E(2)$ (the group of isometries of the Euclidian plane), and these are called Bianchi classes in \cite{IJ92}.

In Bianich classes, given a initial metric $g_0$, there is a Milnor frame $(f_1,f_2,f_3)$ such that the metric and Ricci tensor  are diagonalized.  Since this property is preserved by the Ricci flow, we often write
$$g=Af^1\otimes f^1+Bf^2\otimes f^2+Cf^3\otimes f^3, $$ where $(f^1, f^2, f^3)$ is a dual frame of Milnor frame. Then the Ricci flow reduces to an ODE system involving $A$, $B$ and $C$. Since the homogeneous 3-manifolds are models of geometrization conjecture, Isenberg and Jackson \cite{IJ92} studied the Ricci flow on such manifolds and described their characteristic behaviors by analyzing the corresponding system. There are three types of behaviors depending on the geometry type.

Later Knopf and McLeod \cite {KM01} studied quasi-convergence equivalence of model geometries under the Ricci flow.

If we assume $ g(t)$, $t \in (-T_{b}, T_{f} )$, is the
maximal solution of the forward Ricci flow (1.1), it is interesting to consider the behaviors of $g(t)$  as $t$ goes to $-T_b$. For convenience, we reverse time and consider the following backward Ricci flow equation
 \begin{equation}\frac{\partial g}{\partial t}=2\Rc-\frac{2r}{3}g,\,\,g(0)=g_0.\end{equation}The behaviors of backward Ricci flow are described in \cite{CS} and \cite{CGS}, and the interesting phenomenon is that the backward Ricci flow converges to a sub-Riemannian geometry after a proper rescaling. We use the simple example to explain the term of sub-Riemannian geometry.  In section three, the evolving metric is
$$g=Af^1\otimes f^1+Bf^2\otimes f^2+Cf^3\otimes f^3.$$  By (3.3),  we get
$$A(t)\rightarrow +\infty,\,\,\,B(t)\rightarrow 0,\,\,\, C(t)\rightarrow 0,$$ as $t$ goes to $-3/(16R_0)$. Then the rescaled metric $\bar{g}(t)=(C_0/C(t))g(t)$ converges to
$$\infty f^1\otimes f^1+B_0f^2\otimes f^2+C_0f^3\otimes f^3.$$
Look at  the dual tensor defined on the co-tangent bundle
$$Q=A^{-1}f_1\otimes f_1+B^{-1}f_2\otimes f_2+C^{-1}f_3\otimes f_3.$$
Then the tensor $Q$ tends to
$$Q_{*}=B_0^{-1}f_2\otimes f_2+C_0^{-1} f_3\otimes f_3.$$
It turns out that $[f_2,f_3]=2f_1$. Then the tensor $Q_{*}$ induces a natural distance function $d_*$ defined on $M$.  We take the infimum of the length of all  curves staying tangent to the linear span of $f_2,\,f_3$, which are  called horizontal curves, to compute the distance.  The associated  geometry is called a sub-Riemannian geometry.

In this paper, we study the first eigenvalue  of the Laplace-Beltrami operator under the backward Ricci flow on locally homogeneous 3-manifolds in  Bianchi classes.
\begin{theorem}
Let $(M, g(t)), t\in [0,T_+)$  be a solution to the backward Ricci flow in Bianchi classes, where $T_+\in[0,+\infty]$ is the maximal existence time. Assume that
$\lam(t)$ is the first eigenvalue of $-\Del$ with respect to $g(t)$.  Then in cases where $g(t)$ converges to a sub-Riemannian geometry after a proper re-scaling, $\lam(t)$ goes to zero as $t$ approaches  $T_+$.
\end{theorem}
We construct various monotonic quantities and get upper and lower bounds for the eigenvalue.
The behaviors of the eigenvalue are not very diverse. In many cases, the eigenvalue is decreasing after a time and goes to zero.  The next of this paper is arranged as follows. In section two, we derive an evolution equation of the eigenvalue which is important to estimates. From section three to section seven, we analyze the behaviors of the eigenvalue and get estimates case by case.

\section{Evolution equation of the eigenvalue}
In this section, we get the following theorem.
\begin{theorem}
Let $(M, g(t)), t\in [0,T_+)$  be a solution to the backward Ricci flow  on a locally homogeneous  3-manifold. Assume that
$\lam(t)$ is the first eigenvalue of $-\Del$ and $u(x,t)>0$ satisfies
\begin{equation}-\Delta u=\lam u,\end{equation} with $\int u^2(x,t)d\mu=1.$
Then along  the backward Ricci flow, we obtain
\begin{align}\ddt\lam=\frac{2}{3}R\lam-\int (2R_{ij}\na_iu\na_ju)d\mu.\end{align}
\end{theorem}
The proof of this theorem is similar to Lemma 3.1 in \cite {CHL12}.

\begin{proof}
By a direct calculation as in \cite {Cao07}, we have
\begin{align*}
\ddt\lam=\int\left(2uR_{ij}\na_i\na_ju-\frac{2r}{3}u\Delta u \right)d\mu
\end{align*}
where $$r=\frac{\int_MRd\mu}{\int_Md\mu}$$  is the average of the scalar curvature, which equals  $R$ on locally homogeneous manifolds.
Integrating by parts and using contracted Bianchi identity, we have
\begin{align*}
\int 2uR_{ij}\na_i\na_jud\mu=-\int \left[(2u\na_iR_{ij})\na_ju\right]d\mu-\int (2R_{ij}\na_iu\na_ju)d\mu,
\end{align*}
and
\begin{align*}
-\int (2\na_iR_{ij})u\na_jud\mu &=-\int u(\na_jR)\na_jud\mu=\int Ru\Delta ud\mu+\int R|\na u|^2d\mu\\
&=-\lam R+\lam R=0.
\end{align*}
We arrive at
\begin{align*}
\ddt\lam=\frac{2}{3}R\lam-\int (2R_{ij}\na_iu\na_ju)d\mu.
\end{align*}
\end{proof}
\section{\text{Heisenberg}}
In this class, given a metric $g_0$, there is a fixed Milnor frame such that
$$\left[f_2,f_3\right]=2f_1,\,\,\,\left[f_3,f_1\right]=0,\,\,\,\left[f_1,f_2\right]=0.$$

 Under the normalization $A_0B_0C_0=4$, the curvature components for metrics are (see examples on page 171 in \cite{BLN})
\begin{equation}
\left\{
\begin{aligned}
&R_{11}=\frac{1}{2}A^{3},\\
&R_{22}=-\frac{1}{2}A^{2}B\\
&R_{33}=-\frac{1}{2}A^{2}C,\\
&R=-\frac{1}{2}A^{2}.
\end{aligned}\right.
\end{equation}
The backward Ricci flow equations are then
\begin{equation}
\left\{
\begin{aligned}
&\ddt A=\frac{4}{3}A^3,\\
&\ddt B=-\frac{2}{3}A^2B,\\
&\ddt C=-\frac{2}{3}A^2C,
\end{aligned}
\right.
\end{equation}
 and the solution   is
\begin{equation}
\left\{
\begin{aligned}
&A=A_{0}\left(1+\frac{16}{3}R_{0}t\right)^{-1/2},\\
&B=B_{0}\left(1+\frac{16}{3}R_{0}t\right)^{1/4},\\
&C=C_{0}\left(1+\frac{16}{3}R_{0}t\right)^{1/4},
\end{aligned}\right.
\end{equation}
where $R_{0}=-\frac{1}{2}A_{0}^{2}$.
The metric $\bar{g}(t)=(C_0/C(t))g(t)$ converges to a sub-Riemannian geometry.
\begin{theorem}
Let $\lam(t)$ be the  first eigenvalue of $-\Del$.  Assume that $B_0\geq C_0$. Then
$\lam(t) e^{\int_{0}^{t}\left(-\frac{2}{3}R+2R_{11}\right)dt}$ is nondecreasing  along the backward Ricci flow, and $\lam(t) e^{\int_{0}^{t}\left(-\frac{2}{3}R+2R_{22}\right)dt}$ is nonincreasing. Moreover, we have
$$e^{\frac{3A_0}{4}\left[1-\left(1+\frac{16R_0t}{3}\right)^{-1/2}\right]}\left(1+\frac{16R_0t}{3}\right)^{1/8}\lam(0)\leq \lam(t)\leq\lam(0)e^{\frac{3B_0}{2}\left[1-\left(1+\frac{16R_0t}{3}\right)^{1/4}\right]}\left(1+\frac{16R_0t}{3}\right)^{1/8}$$ for
$t\in[0,-3/(16R_0)).$  As $t$ goes to $-3/(16R_0)$, $\lam(t)$ goes to zero.
\end{theorem}
\begin{proof}
Assume that $B_0\geq C_0$. By (2.2) and (3.1),  have
$$\frac{2}{3}R\lam -2R_{11}\lam\leq\ddt \lam\leq \frac{2}{3}R\lam -2R_{22}\lam.$$

Then
$\lam(t) e^{\int_{0}^{t}\left(-\frac{2}{3}R+2R_{11}\right)dt}$ is nondecreasing  along the backward Ricci flow, and $\lam(t) e^{\int_{0}^{t}\left(-\frac{2}{3}R+2R_{22}\right)dt}$ is nonincreasing.

Integrating from $0$ to $t$ yields
$$e^{\frac{3A_0}{4}\left[1-\left(1+\frac{16R_0t}{3}\right)^{-1/2}\right]}\left(1+\frac{16R_0t}{3}\right)^{1/8}\lam(0)\leq \lam(t)\leq\lam(0)e^{\frac{3B_0}{2}\left[1-\left(1+\frac{16R_0t}{3}\right)^{1/4}\right]}\left(1+\frac{16R_0t}{3}\right)^{1/8}.$$

\end{proof}

\section{$\text{SU}(2)$}

Given a metric $g_0$, we choose a Milnor frame such that
$$\left[f_2,f_3\right]=2f_1,\,\,\,\left[f_3,f_1\right]=2f_2,\,\,\,\left[f_1,f_2\right]=2f_3.$$
 Under the normalization $A_0B_0C_0=4$, then the nonzero components of the Ricci tensor are (see examples on page 171 in \cite{BLN})
\begin{equation}
\left\{
\begin{aligned}
&R_{11}=\frac{1}{2}A[A^{2}-(B-C)^{2}],\\
&R_{22}=\frac{1}{2}B[B^{2}-(A-C)^{2}],\\
&R_{33}=\frac{1}{2}C[C^{2}-(A-B)^{2}],
\end{aligned}\right.
\end{equation}
and the scalar curvature is
\begin{equation}R=\frac{1}{2}[A^{2}-(B-C)^{2}]+\frac{1}{2}[B^{2}-(A-C)^{2}]+\frac{1}{2}[C^{2}-(A-B)^{2}].\end{equation}
The backward Ricci flow equations are
\begin{equation}
\left\{
\begin{aligned}
&\frac{dA}{dt}=-\frac{2}{3}A\left[-A(2A-B-C)+(B-C)^2\right],\\
&\frac{dB}{dt}=-\frac{2}{3}B\left[-B(2B-A-C)+(A-C)^2\right],\\
&\frac{dC}{dt}=-\frac{2}{3}C\left[-C(2C-A-B)+(A-B)^2\right].
\end{aligned} \right.
\end{equation}
Assume that $A_0\geq B_0\geq C_0$. Cao \cite{CS} proved the following theorem.
\begin{theorem}\noindent

(1) If $A_0=B_0=C_0$, then $T_+=\infty$ and $g(t)=g_0$.

(2) If $A_0=B_0>C_0$, then  $T_+=\infty$, $A=B>C$ and, as $t$ goes to infinity,

$A\sim\frac{8}{3}t,\,C\sim\frac{9}{16}t^{-2}$.

(3) If $A_0>B_0\geq C_0$, then $T_+<\infty$, $A>B\geq C$ and there are constants

$\eta_1,\,\eta_2\in(0,\infty)$ such that

$$A\sim\frac{\sqrt{6}}{4}(T_+-t)^{-1/2},\,B\sim\eta_1(T_+-t)^{1/4},\,C\sim\eta_2(T_+-t)^{1/4}.$$

In case (3), $\bar{g}(t)=(B_0/B(t))g(t)$ converges to a sub-Riemannian geometry.
\end{theorem}
In the following,   $\tau$, $c_1$ and $c_2$ denote  constants which may vary from line to line and from section to section.
\begin{theorem}
Let $\lam(t)$ be the  first eigenvalue of $-\Del$. Then we get

(1) If $A_0=B_0=C_0$, then $\lam(t)=\lam(0)$.

(2) If $A_0=B_0>C_0$, there is a time $\tau$ such that
$\lam(t) e^{\int_{\tau}^{t}\left(-\frac{2}{3}R+2R_{11}\right)dt}$ is nondecreasing  along the backward  Ricci flow, and $\lam(t) e^{\int_{\tau}^{t}\left(-\frac{2}{3}R+2R_{33}\right)dt}$ is nonincreasing. We also get the further estimate
$$\lam(\tau)e^{8(\tau-t)}\leq \lam(t)\leq \lam(\tau)\left(\frac{t}{\tau}\right)^{c_1}.$$

(3) If $A_0>B_0\geq C_0$, then there is a time $\tau$ such that
$\lam(t) e^{\int_{\tau}^{t}\left(-\frac{2}{3}R+2R_{11}\right)dt}$ is nondecreasing along the backward  Ricci flow, and $\lam(t) e^{\int_{\tau}^{t}\left(-\frac{2}{3}R+2R_{33}\right)dt}$ is nonincreasing. We get the following estimate
$$\lam(\tau)e^{2c_2\left[ (T_+-\tau)^{-1/2}-(T_+-t)^{-1/2}\right]}\leq \lam(t)\leq \lam(\tau)\left(\frac{T_+-t}{T_+-\tau}\right)^{c_1}.$$ As $t$ goes to $T_+$, $\lam(t)$ approaches $0$.
\end{theorem}
\begin{proof}
\noindent

 (1) If $A_0=B_0=C_0$, then $\lam(t)=\lam(0)$.

(2) If $A_0=B_0>C_0$, then by (4.1) and (2) in Theorem 4.1 we have
$$R_{11}=R_{22}>0,\,\,R_{33}>0$$
  after a time $\tau$.
It also follows that
\begin{align*}
R_{22}-R_{33}&=\frac{1}{2}[B^3-B(A-C)^2-C^3+C(A-B)^2]\\
&=\frac{1}{2}(B-C)(B^2++2BC+C^2-A^2)\\
&>0
\end{align*}
for $t\geq \tau$.
Thus we derive
$$\frac{2}{3}R\lam -2R_{11}\lam\leq\ddt \lam\leq \frac{2}{3}R\lam -2R_{33}\lam$$ with $t\geq \tau$.

Then
$\lam(t) e^{\int_{\tau}^{t}\left(-\frac{2}{3}R+2R_{11}\right)dt}$ is nondecreasing along the backward  Ricci flow, and $\lam(t) e^{\int_{\tau}^{t}\left(-\frac{2}{3}R+2R_{33}\right)dt}$ is nonincreasing.

Moreover, by (2) in Theorem 4.1 we estimate
\begin{align*}
&\frac{2}{3}R-2R_{33}\\
= &\frac{1}{3}(2AB+2BC+2CA-A^2-B^2-C^2)-[C^3-C(A-B)^2]\\
=&\frac{1}{3}(4AC-C^2)-C^3\\
\leq &c_{1}t^{-1}
\end{align*}
and
\begin{align*}
&\frac{2}{3}R-2R_{11}\\
= &\frac{1}{3}(2AB+2BC+2CA-A^2-B^2-C^2)-[A^3-A(B-C)^2]\\
=&\frac{1}{3}(4AC-C^2)-A^3+A(B-C)^2\\
\geq &-8
\end{align*}
for $t\geq \tau$.

Thus we arrive at
$$-8\leq \frac{1}{\lam}\ddt \lam \leq c_{1}t^{-1}.$$
Integrating from $\tau$ to $t$ gives
$$\lam(\tau)e^{8(\tau-t)}\leq \lam(t)\leq \lam(\tau)\left(\frac{t}{\tau}\right)^{c_1}.$$

(3)  If $A_0>B_0\geq C_0$, then by (3) in Theorem 4.1 we have
$$R_{11}>0,\,\,R_{22}<0,\,\,R_{33}<0,$$
after a time $\tau$.
It is easy to see that
\begin{align*}
R_{22}-R_{33}&=\frac{1}{2}[B^3-B(A-C)^2-C^3+C(A-B)^2]\\
&=\frac{1}{2}(B-C)(B^2+2BC+C^2-A^2)\\
&\leq 0
\end{align*}
if $t\geq \tau$.

Thus  we have $$R_{11}>R_{33}\geq R_{22}$$ and
$$\frac{2}{3}R\lam -2R_{11}\lam\leq\ddt \lam\leq \frac{2}{3}R\lam -2R_{22}\lam$$ with $t\geq \tau$.

Then
$\lam(t) e^{\int_{\tau}^{t}\left(-\frac{2}{3}R+2R_{11}\right)dt}$ is nondecreasing along the backward  Ricci flow, and $\lam(t) e^{\int_{\tau}^{t}\left(-\frac{2}{3}R+2R_{22}\right)dt}$ is nonincreasing.

Furthermore, we conclude from the behavior of the metric that
\begin{align*}
&\frac{2}{3}R-2R_{22}\\
= &\frac{1}{3}(2AB+2BC+2CA-A^2-B^2-C^2)-[B^3-B(A-C)^2]\\
\leq &-c_{1}(T_+-t)^{-1}
\end{align*}
and
\begin{align*}
&\frac{2}{3}R-2R_{11}\\
= &\frac{1}{3}(2AB+2BC+2CA-A^2-B^2-C^2)-[A^3-A(B-C)^2]\\
\geq &-c_{2}(T_+-t)^{-3/2}
\end{align*}
for $t\geq \tau$.

Thus we obtain
$$-c_{2}(T_+-t)^{-3/2}\leq \frac{1}{\lam}\ddt \lam \leq -c_{1}(T_+-t)^{-1}.$$
Integrating from $\tau$ to $t$ gives
$$\lam(\tau)e^{2c_2\left[ (T_+-\tau)^{-1/2}-(T_+-t)^{-1/2}\right]}\leq \lam(t)\leq \lam(\tau)\left(\frac{T_+-t}{T_+-\tau}\right)^{c_1}.$$
As $t$ goes to $T_+$, $\lam(t)$ approaches $0$.
\end{proof}

\section{\text{E}(1,1)}
Given a metric $g_0$, we choose a fixed Milnor frame such that
$$\left[f_2,f_3\right]=2f_1,\,\,\,\left[f_3,f_1\right]=0\,\,\,\left[f_1,f_2\right]=-2f_3.$$
 Under the normalization $A_0B_0C_0=4$, the nonzero curvature components of the metric  are
\begin{equation}
\left\{
\begin{aligned}
&R_{11}=\frac{1}{2}A(A^2-C^2),\\
&R_{22}=-\frac{1}{2}B(A+C)^2,\\
&R_{33}=\frac{1}{2}C(C^2-A^2),\\
&R=-\frac{1}{2}(A+C)^{2}.\end{aligned}\right.
\end{equation}
The backward Ricci flow equations are
\begin{equation}
\left\{
\begin{aligned}
&\frac{dA}{dt}=\frac{2}{3}A(2A^2+AC-C^2),\\
&\frac{dB}{dt}=-\frac{2}{3}B(A+C)^2,\\
&\frac{dC}{dt}=\frac{2}{3}C(2C^2+AC-A^2).
\end{aligned} \right.
\end{equation}
Assume that $A_0\geq C_0$. Cao \cite{CS} proved the following theorem.
\begin{theorem}\noindent

(1) If $A_0=C_0$, then $T_+=\frac{3}{32}B_0$ and
$$A(t)=C(t)=\frac{\sqrt{6}}{4}(T_+-t)^{-1/2},\;\;B(t)=\frac{32}{3}(T_+-t),\;\;t\in[0,T_+).$$

(2) If $A_0>C_0$, then $T_+<\infty$, and there exist constants $\eta_1,\eta_2\in(0,\infty)$ such that
$$A\sim\frac{\sqrt{6}}{4}(T_+-t)^{-1/2},\,B(t)\sim\eta_1(T_+-t)^{1/4},\,C(t)\sim\eta_2(T_+-t)^{1/4},$$
as $t$ goes to $T_+$.

In case (2), $\bar{g}(t)=(B_0/B(t))g(t)$ converges to a sub-Riemannian geometry.
\end{theorem}

\begin{theorem}
Let $\lam(t)$ be the  first eigenvalue of $-\Del$. Then we get

(1) If $A_0=C_0$, then
$\lam(t) e^{\int_{0}^{t}\left(-\frac{2}{3}R\right)dt}$ is nondecreasing along the backward Ricci flow, and $\lam(t) e^{\int_{0}^{t}\left(-\frac{2}{3}R+2R_{22}\right)dt}$ is nonincreasing.
Moreover, we get
$$\lam(0)\left(\frac{T_+-t}{T_+}\right)^{1/2}\leq \lam(t)\leq\lam(0)\left(\frac{T_+-t}{T_+}\right)^{1/2}e^{16t}.$$
(2) If $A_0>C_0\geq B_0$, then there is time $\tau$ such that
$\lam(t) e^{\int_{\tau}^{t}\left(-\frac{2}{3}R+2R_{11}\right)dt}$ is nondecreasing along the backward Ricci flow, and $\lam(t) e^{\int_{\tau}^{t}\left(-\frac{2}{3}R+2R_{33}\right)dt}$ is nonincreasing.
Moreover, we have
$$\lam(\tau)e^{2c_2\left[(T_+-\tau)^{-1/2}-(T_+-t)^{-1/2}\right]}\leq \lam(t)\leq \lam(\tau)\left(\frac{T_+-t}{T_+-\tau}\right)^{c_1}.$$
As $t$ goes to $T_+$, $\lam(t)$ approaches $0$.

\end{theorem}
\begin{remark}
In case (2) of  the above theorem, if $B_0> C_0$, we can get the similar estimate.
\end{remark}
\begin{proof}
(1) If $A_0=C_0$,  then by (5.1) we have
$$R_{11}=0,\,\,R_{22}<0,\,\,R_{33}=0.$$

Thus  we have $$R_{11}=R_{33}>R_{22}$$ and
$$\frac{2}{3}R\lam \leq\ddt \lam\leq \frac{2}{3}R\lam -2R_{22}\lam.$$

Then
$\lam(t) e^{\int_{0}^{t}\left(-\frac{2}{3}R\right)dt}$ is nondecreasing along the backward Ricci flow, and $\lam(t) e^{\int_{0}^{t}\left(-\frac{2}{3}R+2R_{22}\right)dt}$ is nonincreasing.

Next, we do further estimates
\begin{align*}
&\frac{2}{3}R-2R_{22}\\
= &-\frac{1}{3}A^2-\frac{1}{3}C^2-\frac{2}{3}AC+BA^2+BC^2+2ABC\\
= &-\frac{1}{2}(T_+-t)^{-1}+16
\end{align*}
and
\begin{align*}
&\frac{2}{3}R=-\frac{1}{2}(T_+-t)^{-1}.
\end{align*}

Thus we arrive at
$$-\frac{1}{2}(T_+-t)^{-1}\leq \frac{1}{\lam}\ddt \lam \leq -\frac{1}{2}(T_+-t)^{-1}+16.$$
Integration from $0$ to $t$ gives
$$\lam(0)\left(\frac{T_+-t}{T_+}\right)^{1/2}\leq \lam(t)\leq\lam(0)\left(\frac{T_+-t}{T_+}\right)^{1/2}e^{16t}.$$

(2) If $A_0>C_0$, then by (5.1) and (2) in Theorem 5.1 we have
$$R_{11}>0,\,\,R_{22}<0,\,\,R_{33}<0 $$ after a time $\tau$.

Assume that $C_0\geq B_0$.  By (5.2) we get
\begin{align*}
\ddt\ln\frac{C}{B}&=2(AC+C^2),
\end{align*}
which implies $C(t)>B(t)$ for all $t>0$.
 It is easy to see that
\begin{align*}
R_{22}-R_{33}&=\frac{1}{2}(CA^2-BA^2-C^3-BC^2-2ABC)\\
&>0
\end{align*}
after a time $\tau$.

So we arrive at
 $$R_{11}>R_{22}>R_{33}$$ and
$$\frac{2}{3}R\lam -2R_{11}\lam\leq\ddt \lam\leq \frac{2}{3}R\lam -2R_{33}\lam$$ with $t\geq \tau$.

Then
$\lam(t) e^{\int_{\tau}^{t}\left(-\frac{2}{3}R+2R_{11}\right)dt}$ is nondecreasing along the backward Ricci flow, and $\lam(t) e^{\int_{\tau}^{t}\left(-\frac{2}{3}R+2R_{33}\right)dt}$ is nonincreasing.

Moreover, we obtain
\begin{align*}
&\frac{2}{3}R-2R_{33}\\
= &-\frac{1}{3}A^2-\frac{1}{3}C^2-\frac{2}{3}AC-C^3+CA^2\\
\leq &-c_{1}(T_+-t)^{-1}
\end{align*}
and
\begin{align*}
&\frac{2}{3}R-2R_{11}\\
= &-\frac{1}{3}A^2-\frac{1}{3}C^2-\frac{2}{3}AC-A^3+AC^2\\
\geq &-c_{2}(T_+-t)^{-3/2}
\end{align*}
with $t\geq \tau$.

Thus we get
$$-c_{2}(T_+-t)^{-3/2}\leq \frac{1}{\lam}\ddt \lam \leq -c_{1}(T_+-t)^{-1}.$$
Integration from $\tau$ to $t$ yields
$$\lam(\tau)e^{2c_2\left[(T_+-\tau)^{-1/2}-(T_+-t)^{-1/2}\right]}\leq \lam(t)\leq \lam(\tau)\left(\frac{T_+-t}{T_+-\tau}\right)^{c_1}.$$
As $t$ goes to $T_+$, $\lam(t)$ approaches $0$.
\end{proof}

\section{$\text{E}(2)$}
Given a metric $g_0$, we choose a  Milnor frame such that
$$\left[f_2,f_3\right]=2f_1,\,\,\,\left[f_3,f_1\right]=2f_2,\,\,\,\left[f_1,f_2\right]=0.$$
Under the normalization $A_0B_0C_0=4$, the nonzero curvature components are

\begin{equation}
\left\{
\begin{aligned}
&R_{11}=\frac{1}{2}A(A^2-B^2),\\
&R_{22}=\frac{1}{2}B(B^2-A^2),\\
&R_{33}=-\frac{1}{2}C(A-B)^2,\\
&R=-\frac{1}{2}(A-B)^{2}.
\end{aligned}\right.
\end{equation}
Then the backward Ricci flow equations are
\begin{equation}
\left\{
\begin{aligned}
&\frac{dA}{dt}=\frac{2}{3}A(2A+B)(A-B),\\
&\frac{dB}{dt}=-\frac{2}{3}B(2B+A)(A-B),\\
&\frac{dC}{dt}=-\frac{2}{3}C(A-B)^2.
\end{aligned} \right.
\end{equation}
Assume that $A_0\geq B_0$. Cao \cite{CS} proved the following theorem.

\begin{theorem}\noindent

(1) If $A_0=B_0$, then $T_+=\infty$, and $g(t)=g_0$ for $t\in[0,+\infty)$.

(2) If $A_0>B_0$, then $T_+<\infty$, there exist two  positive constants $\eta_1,\,\eta_2$ such that
$$A\sim\frac{\sqrt{6}}{4}(T_+-t)^{-1/2},\,B(t)\sim\eta_1(T_+-t)^{1/4},\,C(t)\sim\eta_2(T_+-t)^{1/4},$$

as $t$ goes to $T_+$.

In case (2), $\bar{g}(t)=(B_0/B(t))g(t)$ converges to a sub-Riemannian geometry.
\end{theorem}
We will prove the following theorem.
\begin{theorem}
Let $\lam(t)$ be the  first eigenvalue of $-\Del$. Then we get

(1) If $A_0=B_0$, then $g(t)=g_0$, and $\lam(t)$ is  a constant.

(2) If $A_0>B_0$ and $C_0\geq B_0$,  then there is a time $\tau$ such that
$\lam(t) e^{\int_{\tau}^{t}\left(-\frac{2}{3}R+2R_{11}\right)dt}$ is nondecreasing along the backward Ricci flow, and $\lam(t) e^{\int_{\tau}^{t}\left(-\frac{2}{3}R+2R_{33}\right)dt}$ is nonincreasing.
Moreover, we have
$$\lam(\tau)e^{2c_2\left[(T_+-\tau)^{-1/2}-(T_+-t)^{-1/2}\right]}\leq \lam(t)\leq \lam(\tau)\left(\frac{T_+-t}{T_+-\tau}\right)^{c_1}.$$
As $t$ goes to $T_+$, $\lam(t)$ approaches $0$.
\end{theorem}
\begin{remark}
In case (2) of  the above theorem, if $C_0<B_0$, we can get the similar estimate.
\end{remark}
\begin{proof}
\noindent

(1) If $A_0=B_0$, then $g(t)=g_0$, and $\lam(t)$ is independent of $t$.

(2) If $A_0>B_0$, then by (6.2)  we have
$$\ddt(A-B)=\frac{4}{3}(A-B)(A^2+AB+B^2).$$
So $A-B$ is increasing and $A(t)>B(t)$. This and (6.1) yield
$$R_{11}>0,\,\,R_{22}<0,\,\,R_{33}<0.$$

Assume that $C_0\geq B_0$. From the equations for $B$, $C$ in (6.2), we get
\begin{align*}
\ddt\ln\frac{C}{B}&=\frac{2}{3}[(2B+A)(A-B)-(A-B)^2]\\
&=2B(A-B)
\end{align*}
and conclude that $\frac{C}{B}$ is increasing and $C(t)>B(t)$ for $t>0$.

By (6.1) and (2) in Theorem 6.1, we get
\begin{align*}
R_{22}-R_{33}&=\frac{1}{2}(A-B)(AC-BC-AB-B^2)\\
&=\frac{1}{2}(A-B)[A(C-B)-BC-B^2]\\
&>0
\end{align*}
 if $t\geq \tau$.
So we obtain $R_{11}>R_{22}>R_{33}$
and
$$\frac{2}{3}R\lam -2R_{11}\lam\leq\ddt \lam\leq \frac{2}{3}R\lam -2R_{33}\lam$$ with $t\geq \tau$.

Then
$\lam(t) e^{\int_{\tau}^{t}\left(-\frac{2}{3}R+2R_{11}\right)dt}$ is nondecreasing along the backward Ricci flow, and $\lam(t) e^{\int_{\tau}^{t}\left(-\frac{2}{3}R+2R_{33}\right)dt}$ is nonincreasing.

Further computations show that
\begin{align*}
&\frac{2}{3}R-2R_{33}\\
= &-\frac{1}{3}(A-B)^2+C(A-B)^2\\
=&-\frac{1}{3}A^2-\frac{1}{3}B^2+\frac{2}{3}AB++CA^2+CB^2-2ABC\\
\leq &-c_{1}(T_+-t)^{-1}
\end{align*}
and
\begin{align*}
&\frac{2}{3}R-2R_{11}\\
= &-\frac{1}{3}(A-B)^2-A(A^2-B^2)\\
=&-\frac{1}{3}A^2-\frac{1}{3}B^2+\frac{2}{3}AB-A^3+AB^2\\
\geq &-c_{2}(T_+-t)^{-3/2}
\end{align*}
after a time $\tau$.
Thus we arrive at
$$-c_{2}(T_+-t)^{-3/2}\leq \frac{1}{\lam}\ddt \lam \leq -c_{1}(T_+-t)^{-1}.$$
Integration from $\tau$ to $t$ gives
$$\lam(\tau)e^{2c_2\left[(T_+-\tau)^{-1/2}-(T_+-t)^{-1/2}\right]}\leq \lam(t)\leq \lam(\tau)\left(\frac{T_+-t}{T_+-\tau}\right)^{c_1}.$$
As $t$ goes to $T_+$, $\lam(t)$ approaches $0$.

\end{proof}

\section{$\SL$}
Given a metric $g_0$, we choose a Milnor frame such that
$$\left[f_2,f_3\right]=-2f_1,\,\,\,\left[f_3,f_1\right]=2f_2,\,\,\,\left[f_1,f_2\right]=2f_3.$$
Under the normalization $A_0B_0C_0=4$, the nonzero curvature components are
\begin{equation}
\left\{
\begin{aligned}
&R_{11}=\frac{1}{2}A[A^{2}-(B-C)^{2}],\\
&R_{22}=\frac{1}{2}B[B^{2}-(A+C)^{2}],\\
&R_{33}=\frac{1}{2}C[C^{2}-(A+B)^{2}],\\
&R=\frac{1}{2}[A^{2}-(B-C)^{2}]+\frac{1}{2}[B^{2}-(A+C)^{2}]+\frac{1}{2}[C^{2}-(A+B)^{2}].\end{aligned}\right.
\end{equation}
Then the backward Ricci flow equations are
\begin{equation}
\left\{
\begin{aligned}
&\frac{dA}{dt}=-\frac{2}{3}[-A^2(2A+B+C)+A(B-C)^2],\\
&\frac{dB}{dt}=-\frac{2}{3}[-B^2(2B+A-C)+B(A+C)^2],\\
&\frac{dC}{dt}=-\frac{2}{3}[-C^2(2C+A-B)+C(A+B)^2].
\end{aligned} \right.
\end{equation}
 Under the assumption  $B_0\geq C_0$,  Cao \cite{CS, CGS} proved the following theorem.
\begin{theorem}
The maximal  existence time $T_+$ is finite. Moreover,

(1) If there exists a time $t_0$ such that $A(t_0)\geq B(t_0)$, then
$$A\sim\frac{\sqrt{6}}{4}(T_+-t)^{-1/2},\,B(t)\sim\eta_1(T_+-t)^{1/4},\,C(t)\sim\eta_2(T_+-t)^{1/4}$$ with positive constants $\eta_i$, $i=1,2.$

(2) If there exists a time $t_0$ such that $A(t_0)\leq B(t_0)-C(t_0)$,  then
$$A\sim\eta_1(T_+-t)^{1/4},\,B(t)\sim\frac{\sqrt{6}}{4}(T_+-t)^{-1/2},\,C(t)\sim\eta_2(T_+-t)^{1/4}$$ with positive constants $\eta_i$, $i=1,2.$

(3) If $A<B<A+C$ for all time $t\in [0,T_+)$, we arrive at
$$A\sim\frac{\sqrt{6}}{4}(T_+-t)^{-1/2},\,B(t)\sim\frac{\sqrt{6}}{4}(T_+-t)^{-1/2},\,C(t)\sim\frac{32}{3}(T_+-t).$$

In all cases, the metric $g(t)$ converges to a sub-Riemannian geometry after a proper rescaling.
\end{theorem}
We have the following theorem.
\begin{theorem}
Let $\lam(t)$ be the  first eigenvalue of $-\Del$. Then we get

(1) If there is a time $t_0$ such that $A(t_0)\geq B(t_0)$, then  there exists a  time $\tau$ such that
$\lam(t) e^{\int_{\tau}^{t}\left(-\frac{2}{3}R+2R_{11}\right)dt}$ is nondecreasing  along the backward Ricci flow, and $\lam(t) e^{\int_{\tau}^{t}\left(-\frac{2}{3}R+2R_{22}\right)dt}$ is nonincreasing.
Moreover, we have
$$\lam(\tau)e^{2c_2\left[(T_+-\tau)^{-1/2}-(T_+-t)^{-1/2}\right]}\leq \lam(t)\leq \lam(\tau)\left(\frac{T_+-t}{T_+-\tau}\right)^{c_1}.$$
As $t$ goes to $T_+$, $\lam(t)$ approaches $0$.

(2)  If there exist a time $t_0$ such that $A(t_0)\leq B(t_0)-C(t_0)$, and a time $t_1$ such that $A(t_1)>C(t_1)$, then there is a time $\tau$ such that
$\lam(t) e^{\int_{\tau}^{t}\left(-\frac{2}{3}R+2R_{22}\right)dt}$ is nondecreasing  along the backward  Ricci flow, and $\lam(t) e^{\int_{\tau}^{t}\left(-\frac{2}{3}R+2R_{11}\right)dt}$ is nonincreasing.
Moreover, we have
$$\lam(\tau)e^{2c_2\left[(T_+-\tau)^{-1/2}-(T_+-t)^{-1/2}\right]}\leq \lam(t)\leq \lam(\tau)\left(\frac{T_+-t}{T_+-\tau}\right)^{c_1},$$
and $\lam(t)$ approaches $0$ as $t$ goes to $T_+$.

(3) If $ A<B < A + C$ for all time $t \in [0, T_+)$,  then there is time $\tau$ such that
$\lam(t) e^{\int_{\tau}^{t}\left(-\frac{2}{3}R+2R_{11}\right)dt}$ is nondecreasing along the backward Ricci flow, and $\lam(t) e^{\int_{\tau}^{t}\left(-\frac{2}{3}R+2R_{33}\right)dt}$ is nonincreasing.
Moreover, we have
$$\lam(\tau)\left(\frac{T_+-t}{T_+-\tau}\right)^{c_2}\leq \lam(t)\leq \lam(\tau)\left(\frac{T_+-t}{T_+-\tau}\right)^{c_1}.$$
Obviously,  $\lam(t)$ approaches $0$, as $t$ goes to $T_+$.

\end{theorem}
\begin{remark}
If $A(t)\leq C(t)$ for all $t$ in case (2) of the above theorem,  we can get the similar estimate.
\end{remark}

\begin{proof}
(1) If there is a time $t_0$ such that $A(t_0)>B_{0}$, then by (7.1)  and (1) in Theorem 7.1,  we have
$$R_{11}>0,\,\,R_{22}<0,\,\,R_{33}<0$$ after a time $\tau$.

Next we compare $R_{22}$ with $R_{33}$:

\begin{align*}
R_{22}-R_{33}&=\frac{1}{2}B[B^{2}-(A+C)^{2}]-\frac{1}{2}C[C^{2}-(A+B)^{2}]\\
&=\frac{1}{2}(B-C)[(B+C)^{2}-A^{2}].
\end{align*}

The evolution equation of $B-C$ is
$$\ddt(B-C)=\frac{2}{3}[2(B^3-C^3)+A(B^2-C^2)-A^2(B-C)],$$ from which it follows that $B\geq C$ for all $t$.

Thus  we obtain $$R_{11}>R_{33}\geq R_{22}$$ and
$$\frac{2}{3}R\lam -2R_{11}\lam\leq\ddt \lam\leq \frac{2}{3}R\lam -2R_{22}\lam$$ with $t\geq \tau$.

Then
$\lam(t) e^{\int_{\tau}^{t}\left(-\frac{2}{3}R+2R_{11}\right)dt}$ is nondecreasing  along the backward  Ricci flow, and $\lam(t) e^{\int_{\tau}^{t}\left(-\frac{2}{3}R+2R_{22}\right)dt}$ is nonincreasing.

Moreover, we get\begin{align*}
&\frac{2}{3}R-2R_{22}\\
= &\frac{1}{3}(2BC-2AB-2AC-A^2-B^2-C^2)-B[B^{2}-(A+C)^{2}]\\
\leq & -c_1(T_+-t)^{-1}
\end{align*}
and
\begin{align*}
&\frac{2}{3}R-2R_{11}\\
= &\frac{1}{3}(2BC-2AB-2AC-A^2-B^2-C^2)-A[A^{2}-(B-C)^{2}]\\
\geq & -c_2(T_+-t)^{-3/2}
\end{align*}
after a time $\tau$.

Integrating from $\tau$ to $t$ gives
$$\lam(\tau)e^{2c_2\left[(T_+-\tau)^{-1/2}-(T_+-t)^{-1/2}\right]}\leq \lam(t)\leq \lam(\tau)\left(\frac{T_+-t}{T_+-\tau}\right)^{c_1}.$$
As $t$ goes to $T_+$, $\lam(t)$ approaches $0$.

(2) If there exists a time $t_0$ such that $A(t_0)\leq B(t_0)-C(t_0)$ , the second case in Theorem 7.1 implies
$$R_{11}<0,\,\,R_{22}>0,\,\,R_{33}<0$$ after a time $\tau$.

It follows from  equations (7.2) that
$$\ddt \ln(A/C)=2(A+B-C)(A+C),$$ which implies that $\ddt \ln(A/C)$ is increasing since $B\geq C$ is preserved.

Assume that there is a time $t_1$ such that $A(t_1)> C(t_1)$. Then $A(t)>C(t)$ after $t_1$.

The behaviors of $A$, $B$ and $C$ yield that
\begin{align*}
R_{11}-R_{33}&=\frac{1}{2}A(A^{2}-(B-C)^{2}]-\frac{1}{2}C[C^{2}-(A+B)^{2}]\\
&=\frac{1}{2}(C-A)B^2+\frac{1}{2}(A^3-C^3+CA^2-AC^2+4ABC)\\
&< 0
\end{align*}
after a time $\tau$.

Thus we arrive at
 $$R_{22}>R_{33}>R_{11}$$ and
$$\frac{2}{3}R\lam -2R_{22}\lam<\ddt \lam< \frac{2}{3}R\lam -2R_{11}\lam$$ with $t\geq \tau$.

Then
$\lam(t) e^{\int_{\tau}^{t}\left(-\frac{2}{3}R+2R_{22}\right)dt}$ is nondecreasing along the backward  Ricci flow, and $\lam(t) e^{\int_{\tau}^{t}\left(-\frac{2}{3}R+2R_{11}\right)dt}$ is nonincreasing.

Now we estimate
\begin{align*}
&\frac{2}{3}R-2R_{22}\\
= &\frac{1}{3}(2BC-2AB-2AC-A^2-B^2-C^2)-B[B^{2}-(A+C)^{2}]\\
\geq &-c_2(T_+-t)^{-3/2}
\end{align*}
and
\begin{align*}
&\frac{2}{3}R-2R_{11}\\
= &\frac{1}{3}(2BC-2AB-2AC-A^2-B^2-C^2)-A[A^{2}-(B-C)^{2}]\\
\leq &-c_1(T_+-t)^{-1}
\end{align*}
if $t\geq \tau$.

The above estimates imply
$$\lam(\tau)e^{2c_2\left[(T_+-\tau)^{-1/2}-(T_+-t)^{-1/2}\right]}\leq \lam(t)\leq \lam(\tau)\left(\frac{T_+-t}{T_+-\tau}\right)^{c_1},$$
and $\lam(t)$ approaches $0$, as $t$ goes to $T_+$.

(3) If $A<B<A+C$ for all time $t\in [0,T_+)$, it follows from (7.1) and Theorem 7.1 that
$$R_{11}>0,\,\,R_{22}<0,\,\,R_{33}<0$$ after a time $\tau$.

Since $A<B<A+C$, we obtain
 \begin{align*}
R_{22}-R_{33}&=\frac{1}{2}(B-C)[(B+C)^{2}-A^{2}]\geq 0.
\end{align*}
Thus  we have $$R_{11}>R_{22}\geq R_{33}$$ and
$$\frac{2}{3}R\lam -2R_{11}\lam<\ddt \lam\leq \frac{2}{3}R\lam -2R_{33}\lam$$ with $t\geq \tau$.

Thus
$\lam(t) e^{\int_{\tau}^{t}\left(-\frac{2}{3}R+2R_{11}\right)dt}$ is nondecreasing along the backward Ricci flow, and $\lam(t) e^{\int_{\tau}^{t}\left(-\frac{2}{3}R+2R_{33}\right)dt}$ is nonincreasing.

Moreover, we compute
\begin{align*}
&\frac{2}{3}R-2R_{33}\\
= &\frac{1}{3}(2BC-2AB-2AC-A^2-B^2-C^2)-C[C^{2}-(A+B)^{2}]\\
\leq &-c_1(T_+-t)^{-1}
\end{align*}

and
\begin{align*}
&\frac{2}{3}R-2R_{11}\\
= &\frac{1}{3}(2BC-2AB-2AC-A^2-B^2-C^2)-A[A^{2}-(B-C)^{2}]\\
= &\frac{1}{3}(2BC-2AB-2AC-A^2-B^2-C^2)+A[A+(B-C)][B-A-C]\\
\geq & -c_2(T_+-t)^{-1}
\end{align*}
after a time  $\tau$ since in this case the fact $\lim_{T_+}C=0$ yields $\lim_{T_+}(B-A)=0$.

After integrating  from $\tau$ to $T_+$, we get
$$\lam(\tau)\left(\frac{T_+-t}{T_+-\tau}\right)^{c_2}\leq \lam(t)\leq \lam(\tau)\left(\frac{T_+-t}{T_+-\tau}\right)^{c_1}.$$
Obviously,  $\lam(t)$ approaches $0$, as $t$ goes to $T_+$.
\end{proof}

\vskip 30 pt

{\bf Acknowledgement}

 The author would like to thank Professor Xiaodong Cao and Professor Laurent Saloff-Coste for their suggestions and interests in this work.   The author would also like to thank referees for their valuable comments. The author was supported by NSFC (11001268) and Chinese Universities Scientific Fund (2014QJ002).

 \end{document}